\documentclass[twoside,12pt]{amsart}
\usepackage{fullpage}
\usepackage[safe]{tipa}

\numberwithin{equation}{section}
\newtheorem{Lemma}[equation]{Lemma}
\newtheorem{Theorem}[equation]{Theorem}
\newtheorem{Corollary}[equation]{Corollary}

\theoremstyle{definition}
\newtheorem{Remark}[equation]{Remark}

\def\iso{\cong}

\def\sub{\subseteq}

\def\bar{\overline}
\def\part{{\operatorname{part}}}
\def\Ann{\operatorname{Ann}}
\def\RS{\operatorname{RS}}

\def\C{{\mathbb C}}
\def\Z{{\mathbb Z}}

\def\row{\operatorname{row}}
\def\Sp{\operatorname{Sp}}
\def\SO{\operatorname{SO}}
\def\O{\operatorname{O}}

\def\b{\mathfrak b}
\def\g{\mathfrak g}
\def\q{\mathfrak q}
\def\s{\mathfrak s}
\def\t{\mathfrak t}
\def\u{\mathfrak u}
\def\gl{\mathfrak{gl}}
\def\sl{\mathfrak{sl}}
\def\so{\mathfrak{so}}
\def\sp{\mathfrak{sp}}

\def\cL{\mathcal L}

\def\bp{{\mathbf p}}

\def\sign{{\operatorname{sign}}}
\def\Tab{{\operatorname{Tab}}}
\def\sTab{{\operatorname{sTab}}}
\def\Row{{\operatorname{Row}}}
\def\sRow{{\operatorname{sRow}}}
\def\word{{\operatorname{word}}}
\def\part{{\operatorname{part}}}
\def\eps{\epsilon}

\def\cop{\cdot_{-}}
\def\dop{\cdot_{+}}
\def\phiop{\cdot_{\phi}}

\title{
Non-integral representation theory of even multiplicity
finite $W$-algebras
}
\author
{Jonathan Brown and Simon M.~Goodwin}

\address{School of Mathematics, University of Birmingham, Birmingham, B15
2TT,~UK}
\email{brownjs@for.mat.bham.ac.uk, goodwin@for.mat.bham.ac.uk}


\begin{document}

\begin{abstract}
We complete the classification of the finite dimensional irreducible representations of finite
$W$-algebras associated
to even multiplicity nilpotent elements in classical
Lie algebras.  This
extends earlier work
in \cite{BG1}, where this classification is determined for such representations of integral
central character.
\end{abstract}

\maketitle

\section{Introduction}
Let $U(\g,e)$ denote the finite $W$-algebra associated to a reductive Lie algebra
$\g$ over $\C$ and a nilpotent element $e \in \g$.  In general there is little concrete knowledge about
the representation theory of $U(\g,e)$, except in some special cases.
For $\g = \gl_n$ a thorough study of the representation theory
of $U(\g,e)$ was completed by Brundan and Kleshchev in \cite{BK2}; in particular they obtained
a classification of finite dimensional irreducible $U(\g,e)$-modules.
In \cite{BG1} the authors classified the finite dimensional irreducible
$U(\g,e)$-modules of integral central character when $\g = \so_{2n}$ or $\sp_{2n}$ and $e$ is an {\em even multiplicity} nilpotent element, i.e.\ all the parts of
the Jordan type of $e$ have even multiplicity.
The main result of this paper is Theorem~\ref{T:main}, which
completes the classification of finite dimensional irreducible representations of $U(\g,e)$
for such $\g$ and $e$.

The key ingredients for the proof of this theorem are: the
the classification of finite dimensional irreducible $U(\g,e)$-modules
when $\g$ is of type A
from \cite{BK2}; the
classification in \cite{Bro2} of
irreducible $U(\g,e)$-modules
when $\g$ is of classical type and $e$ is
a {\em rectangular} nilpotent element, i.e.\ the Jordan type of $e$
has all parts equal;
the highest weight theory for finite $W$-algebras developed in \cite{BGK};
the classification for integral central characters from \cite{BG1}; the 
combinatorics for changing highest weight theories from \cite{BG2};
and Theorem~\ref{T:simon} below, which gives an inductive criterion for
an irreducible highest weight $U(\g,e)$-module to be finite dimensional.

To state Theorem~\ref{T:main}, we
first need a specific realization of the classical Lie algebras.
Let $\g = \so_{2n}$ or $\sp_{2n}$.
Let $\gl_{2n} = \langle e_{i,j} \mid i,j = \pm 1, \dots, \pm n \rangle$.
For $i, j = \pm 1, \dots, \pm n$ let
$f_{i,j} = e_{i,j} - e_{-j,-i}$
in the case that $\g = \so_{2n}$,
and let $f_{i,j} = e_{i,j} - \sign(i) \sign(j) e_{-j,-i}$  in
the case that $\g = \sp_{2n}$.
Now we realize $\g$ as the subalgebra of $\gl_{2n}$ spanned by
$\{ f_{i,j} \mid i,j =\pm 1, \dots, \pm n\}$.
We choose $\t = \langle f_{i,i} \rangle$ as a
maximal toral subalgebra of $\g$.
Then $\t^* = \langle \eps_i \mid i = 1,2,\dots,n \rangle$, where $\eps_i$ is dual to $f_{i,i}$.

We recall the definitions
of {\em s-frames}, {\em s-tables}, and {\em symmetric pyramids} from \cite[\S3]{BG1}.
An {\em s-frame} is a connected collection of rows of
boxes in
the plane which is symmetric with respect
to the $x$ and $y$ axes.
We label the rows of an s-frame with $2r$ rows with $1, \dots, r, -r, \dots, -1$
from top to bottom.
A {\em symmetric pyramid} is an s-frame whose rows weakly decrease in length as one moves away from the $x$-axis.
Given a symmetric pyramid $P$ with $2n$ boxes, we define $K$, the {\em coordinate
pyramid for $P$},
to be the filling of $P$ with the numbers
$1, \dots, n, -n, \dots -1$ across rows from top to bottom.
An {\em s-table} is skew symmetric filling of an s-frame with complex numbers.

For example,
\begin{equation*}
P=
    \begin{array}{c}
\begin{picture}(60,80)
\put(10,0){\line(1,0){40}} \put(0,20){\line(1,0){60}}
\put(0,40){\line(1,0){60}} \put(0,60){\line(1,0){60}}
\put(10,80){\line(1,0){40}} \put(10,0){\line(0,1){20}}
\put(30,0){\line(0,1){20}} \put(50,0){\line(0,1){20}}
\put(0,20){\line(0,1){40}} \put(20,20){\line(0,1){40}}
\put(40,20){\line(0,1){40}}
\put(60,20){\line(0,1){40}}
\put(10,60){\line(0,1){20}} \put(30,60){\line(0,1){20}}
\put(50,60){\line(0,1){20}}
\put(30,40){\circle*{3}}
\end{picture}
\end{array}
\end{equation*}
is a symmetric pyramid with coordinate pyramid
\begin{equation*}
K =
    \begin{array}{c}
\begin{picture}(60,80)
\put(10,0){\line(1,0){40}} \put(0,20){\line(1,0){60}}
\put(0,40){\line(1,0){60}} \put(0,60){\line(1,0){60}}
\put(10,80){\line(1,0){40}} \put(10,0){\line(0,1){20}}
\put(30,0){\line(0,1){20}} \put(50,0){\line(0,1){20}}
\put(0,20){\line(0,1){40}} \put(20,20){\line(0,1){40}}
\put(40,20){\line(0,1){40}}
\put(60,20){\line(0,1){40}}
\put(10,60){\line(0,1){20}} \put(30,60){\line(0,1){20}}
\put(50,60){\line(0,1){20}}
\put(30,40){\circle*{3}}
\put(20,70){\makebox(0,0){1}}
\put(40,70){\makebox(0,0){2}}
\put(10,50){\makebox(0,0){3}}
\put(30,50){\makebox(0,0){4}}
\put(50,50){\makebox(0,0){5}}
\put(8,30){\makebox(0,0){-5}}
\put(28,30){\makebox(0,0){-4}}
\put(48,30){\makebox(0,0){-3}}
\put(18,10){\makebox(0,0){-2}}
\put(38,10){\makebox(0,0){-1}}
\end{picture}
\end{array},
\end{equation*}
and
\begin{equation*}
A=
    \begin{array}{c}
\begin{picture}(60,80)
\put(10,0){\line(1,0){40}} \put(0,20){\line(1,0){60}}
\put(0,40){\line(1,0){60}} \put(0,60){\line(1,0){60}}
\put(10,80){\line(1,0){40}} \put(10,0){\line(0,1){20}}
\put(30,0){\line(0,1){20}} \put(50,0){\line(0,1){20}}
\put(0,20){\line(0,1){40}} \put(20,20){\line(0,1){40}}
\put(40,20){\line(0,1){40}}
\put(60,20){\line(0,1){40}}
\put(10,60){\line(0,1){20}} \put(30,60){\line(0,1){20}}
\put(50,60){\line(0,1){20}}
\put(30,40){\circle*{3}}
\put(20,70){\makebox(0,0){$a$}}
\put(40,70){\makebox(0,0){$b$}}
\put(10,50){\makebox(0,0){$c$}}
\put(30,50){\makebox(0,0){$d$}}
\put(50,50){\makebox(0,0){$f$}}
\put(8,30){\makebox(0,0){-$f$}}
\put(28,30){\makebox(0,0){-$d$}}
\put(48,30){\makebox(0,0){-$c$}}
\put(18,10){\makebox(0,0){-$b$}}
\put(38,10){\makebox(0,0){-$a$}}
\end{picture}
\end{array}
\end{equation*}
is an s-table with underling s-frame $P$, where $a,b,c,d,f \in \C$.

Fix a symmetric pyramid $P$ with coordinate pyramid $K$.
We define the nilpotent element $e \in \g$ via $e = \sum f_{i,j}$, where we sum over all $i,j \in \{1, \dots, n\}$
such that $i$ is the left neighbor of $j$ in $K$.
So the Jordan type of $e$ has parts given by the row lengths as $P$,
and  $e$ is an even multiplicity nilpotent element of $\g$.
For example,
if $K$ is as above, then $e = f_{1,2} + f_{3,4} + f_{4,5}$.  We note
that a representative of any even multiplicity nilpotent $\tilde G$-orbit in $\g$ can be obtained in this way,
where $\tilde G = \O_{2n}$  if $\g =
\so_{2n}$ and $\tilde G = \Sp_{2n}$ if
$\g = \sp_{2n}$.

We let $\sTab(P)$
denote the set of s-tables with underlying frame $P$.
We identify $\sTab(P)$ with $\t^*$ via
$\sum_{i=1}^n a_i \eps_i \leftrightarrow A$, where $A \in \sTab(P)$ has $a_i$
in the same box as $i$ in $K$.
For example, if $A$ is as above, then
$A$ corresponds to the weight
$a \eps_1 + b \eps_2 + c \eps_3 + d \eps_4 + f \eps_5$.
We say two s-tables are {\em row equivalent} if one can be obtained from the
other by permuting entries within rows.
We let $\sRow(P)$ denote the set of row equivalence classes of elements of
$\sTab(P)$.
For $A \in \sTab(P)$ we let $\bar A \in \sRow(P)$ denote the row equivalence
class containing $A$.

Let $\t^e = \{t \in \t \mid [e,t] = 0\}$,
and let $\g_0 = \g^{\t^e} = \{ x \in \g \mid [x,\t^e] = \{0\}\}$.
So $\g_0$ is a Levi subalgebra of $\g$, and $e$ is a regular nilpotent element of $\g_0$.
We have
$$
\t^e = \left\langle \sum_{j \in \row(i)} f_{j,j} \:\: \vline \:\: i = 1,\dots,r \right\rangle
$$
and
$$
\g_0 = \left\langle f_{i,j} \mid \row(i) = \row(j) \right\rangle \iso \bigoplus_{i=1}^r \gl_{m_i} ,
$$
where $\row(i)$ denotes the row of $K$ in which $i$ occurs and $m_i$ is the length of row $i$
of $P$.
We choose $\q$, a parabolic subalgebra of $\g$ which contains $\g_0$ as its Levi factor,
by setting $\q = \langle f_{i,j} \mid \row(i) \leq \row(j)\rangle$.

By a theorem of Kostant (\cite{Ko}) we have that
$U(\g_0,e) \cong Z(\g_0) \cong S(\t)^{W_0}$, where
$Z(\g_0)$ denotes the center of $U(\g)$ and $W_0 \iso \prod_{i=1}^r S_{m_i}$ is the Weyl group of $\g_0$.  Now we
identify
$\sRow(P)$ with $\t^*/W_0$,
so
$\sRow(P)$ parameterizes the
irreducible finite dimensional $U(\g_0,e)$-modules.
For $\bar A \in \sRow(P)$ we let $L(\bar A,\q)$
denote the irreducible highest weight $U(\g,e)$-module corresponding to $\bar A$ and $\q$,
as defined using the highest weight theory from
\cite{BGK}.

Given $A \in \sTab(P)$ and $z \in \C$, we define $A_z$ to be the diagram
with the same number of rows of $A$ and whose entries in a given row are the elements
of the corresponding row of $A$ which belong to $z + \Z \in \C / \Z$; by a diagram we just mean
a collection of boxes in the plane.
Note that $A_z$ can have empty rows.
For example,
if
\begin{equation*}
A =
    \begin{array}{c}
\begin{picture}(60,80)
\put(10,0){\line(1,0){40}} \put(0,20){\line(1,0){60}}
\put(0,40){\line(1,0){60}} \put(0,60){\line(1,0){60}}
\put(10,80){\line(1,0){40}} \put(10,0){\line(0,1){20}}
\put(30,0){\line(0,1){20}} \put(50,0){\line(0,1){20}}
\put(0,20){\line(0,1){40}} \put(20,20){\line(0,1){40}}
\put(40,20){\line(0,1){40}}
\put(60,20){\line(0,1){40}}
\put(10,60){\line(0,1){20}} \put(30,60){\line(0,1){20}}
\put(50,60){\line(0,1){20}}
\put(30,40){\circle*{3}}
\put(20,70){\makebox(0,0){$4$}}
\put(40,70){\makebox(0,0){$5$}}
\put(10,50){\makebox(0,0){$\pi$}}
\put(30,50){\makebox(0,0){$2$}}
\put(50,50){\makebox(0,0){$3$}}
\put(8,30){\makebox(0,0){-$3$}}
\put(28,30){\makebox(0,0){-$2$}}
\put(48,30){\makebox(0,0){-$\pi$}}
\put(18,10){\makebox(0,0){-$5$}}
\put(38,10){\makebox(0,0){-$3$}}
\end{picture}
\end{array},
\end{equation*}
then
\[
A_{0} =
    \begin{array}{c}
\begin{picture}(40,80)
\put(0,0){\line(1,0){40}}
\put(0,20){\line(1,0){40}}
\put(0,40){\line(1,0){40}}
\put(0,60){\line(1,0){40}}
\put(0,80){\line(1,0){40}}
\put(0,0){\line(0,1){80}}
\put(20,0){\line(0,1){80}}
\put(40,0){\line(0,1){80}}
\put(10,70){\makebox(0,0){$4$}}
\put(30,70){\makebox(0,0){$5$}}
\put(10,50){\makebox(0,0){$2$}}
\put(30,50){\makebox(0,0){$3$}}
\put(8,30){\makebox(0,0){-$3$}}
\put(28,30){\makebox(0,0){-$2$}}
\put(8,10){\makebox(0,0){-$5$}}
\put(28,10){\makebox(0,0){-$3$}}
\end{picture}
\end{array}
\quad \text{ and } \quad
A_{\pi} =
    \begin{array}{c}
\begin{picture}(20,80)
\put(0,40){\line(1,0){20}}
\put(0,60){\line(1,0){20}}
\put(0,40){\line(0,1){20}}
\put(20,40){\line(0,1){20}}
\put(10,50){\makebox(0,0){$\pi$}}
\end{picture}
\end{array}.
\]

We use the partial order $\le$ on $\C$ defined by $a \le b$ if $b-a \in \Z_{\ge 0}$.
A diagram with boxes aligned in columns is said to be {\em column strict} if all entries are strictly decreasing down columns with respect to $\le$.
We say that a diagram is {\em justified row equivalent to column strict} if when all of the rows of
the diagram are left justified, then the resulting diagram is row equivalent to a column strict
diagram. A diagram is said to be {\em convex} if when we left justify the rows, then the resulting diagram has connected columns.

We define $\sTab^\diamond(P)$ to be the set of $ A \in \sTab(P)$
which satisfy the following conditions:
\begin{itemize}
\item[(M1)]
      For every row of $A$, there
     is at most one
      element $\pm y + \Z$ of the set $\{\pm z + \Z \mid z \in \C \}$
      such that the row contains an odd number of entries from $\pm y + \Z$.
 \item[(M2)]
       For each $z + \Z \in \C / \Z$
       and for every row of $A$
the number of entries in that row
       from $z + \Z$ must differ from the number of entries from $-z+\Z$ by at most one.
\item[(M3)]
       For each
         element $\pm y + \Z$ of the set $\{\pm z + \Z \mid z \in \C \setminus \frac{1}{2} \Z \}$
         and for each odd integer $k$ there must be at most two rows of $A$ with
         precisely $k$ entries from $\pm y + \Z$, i.e.\ at most one such row in the upper half-plane.
\item[(M4)]
$A_z$ is convex and justified row equivalent to column strict
for each $z \in \C$.
\end{itemize}
For example,
\begin{equation*}
    \begin{array}{c}
\begin{picture}(60,80)
\put(10,0){\line(1,0){40}} \put(0,20){\line(1,0){60}}
\put(0,40){\line(1,0){60}} \put(0,60){\line(1,0){60}}
\put(10,80){\line(1,0){40}} \put(10,0){\line(0,1){20}}
\put(30,0){\line(0,1){20}} \put(50,0){\line(0,1){20}}
\put(0,20){\line(0,1){40}} \put(20,20){\line(0,1){40}}
\put(40,20){\line(0,1){40}}
\put(60,20){\line(0,1){40}}
\put(10,60){\line(0,1){20}} \put(30,60){\line(0,1){20}}
\put(50,60){\line(0,1){20}}
\put(30,40){\circle*{3}}
\put(20,70){\makebox(0,0){$4$}}
\put(40,70){\makebox(0,0){$5$}}
\put(10,50){\makebox(0,0){$\pi$}}
\put(30,50){\makebox(0,0){$2$}}
\put(50,50){\makebox(0,0){$3$}}
\put(8,30){\makebox(0,0){-$3$}}
\put(28,30){\makebox(0,0){-$2$}}
\put(48,30){\makebox(0,0){-$\pi$}}
\put(18,10){\makebox(0,0){-$5$}}
\put(38,10){\makebox(0,0){-$4$}}
\end{picture}
\end{array}
\in \sTab^\diamond(P).
\end{equation*}

Given an s-table $A$ with $2 r$ rows, let $A^+$ denote an s-table where an extra
box is added to each row of $A$, with entry $i-1$ in the $i$th row.
Since we are only
concerned with
$A^+$ up to row equivalence, we do not specify precisely how these boxes are
inserted.
We use the table $A^+$ as it gives a convenient way to state the part of
Theorem~\ref{T:main} pertaining to $\g = \sp_{2n}$, though this is not entirely necessary, see Remark~\ref{R:noplus}.
Now we define $\sTab^{\diamond,+} (P)$ to be $\{ A \in \sTab(P) \mid A^+ \in \sTab^\diamond(P^+)\}$,
where $P^+$ is the symmetric pyramid with one more box in each row than $P$.

Finally define $\sRow^\diamond(P)$ and $\sRow^{\diamond,+}(P)$ to be the elements of $\sRow(P)$ which
contain elements of $\sTab^\diamond(P)$ or $\sTab^{\diamond,+}(P)$ respectively.

Let $\tilde C =  C_{\tilde G}(e)/ C_{\tilde G}(e)^\circ$ be the component group of the centralizer of $e$ in $\tilde G$.
In \cite[\S5.5]{BG2} we defined an action of $\tilde C$ on
$\{ \bar A \in \sRow(P) \mid L(\bar A,\q) \text{ is finite
dimensional}\}$.
We note that although the definition of this action of $\tilde C$ in
\cite{BG2} is only defined on
s-tables with entries in $\frac{1}{2} \Z$, this definition works equally well on
s-tables with complex entries.  Furthermore, the results in \cite{BG2} that
are stated only for s-tables with entries in entries in $\frac{1}{2} \Z$
remain valid for arbitrary complex entries, and we use these results
in this paper.

We are now in a position to state our main theorem.

\begin{Theorem} \label{T:main} $ $
\begin{enumerate}
\item[(a)] Let $\g = \so_{2n}$. Then
\[
    \{ L(\bar A,\q) \mid \bar A \in \sRow(P), \bar A \text{ is $\tilde C$-conjugate
to an element of }
      \sRow^\diamond(P) \}
\]
is a complete set of isomorphism classes of finite dimensional irreducible
$U(\g,e)$-modules.
\item[(b)]
Let $\g = \sp_{2n}$.  Then
\[
    \{ L(\bar A,\q) \mid \bar A \in \sRow(P), \bar A \text{ is $\tilde C$-conjugate
to an element of }
      \sRow^{\diamond,+}(P) \}
\]
is a complete set of isomorphism classes of finite dimensional irreducible
$U(\g,e)$-modules.
\end{enumerate}
\end{Theorem}

\begin{Remark} \label{R:noplus}
We note that in the case that $\g = \sp_{2n}$
it is possible to state the classification in Theorem~\ref{T:main}(b)
without using the table $A^+$.  To do
this we replace the condition (M1) with the following condition, which only applies when $\g = \sp_{2n}$:
\begin{itemize}
 \item[(M1$'$)]
    If a row of $A$ has an even number of integral entries, then the number of entries in that row from
      every element $\pm y + \Z$ of the set $\{\pm z + \Z \mid z \in \C \setminus \Z \}$
       must be even.
    If a row of $A$ has an odd number of integral entries, then there must be at most one
      element $\pm y + \Z$ of the set $\{\pm z + \Z \mid z \in \C \setminus \Z \}$
      with an odd number of entries in that row.
\end{itemize}
The statement we use is more convenient for us, in particular, for the proof.
\end{Remark}

As a corollary to this theorem we obtain a parametrization of the primitive ideals of $U(\g)$ with associated variety
$\overline {G \cdot e}$, where $G$ denotes the adjoint group of $\g$.  Below we only state this for the case $\g = \sp_{2n}$, as the situation
for $\so_{2n}$ is more complicated, see Corollary~\ref{C:prim} for the full statement.
For every element of the set $\{\pm z  + \Z \mid z \in \C \setminus \frac{1}{2} \Z\}$ make a choice
$z_+ + \Z$ and $z_- + \Z$ so that
$\pm z + \Z = (z_+ + \Z) \sqcup (z_- + \Z)$.
Let $\sRow^{\bullet,+}(P)$ denote the set of elements $\bar A \in \sRow^{\diamond,+}(P)$
such that
for each element of the set $\{\pm z + \Z \mid z \in \C \setminus \frac{1}{2} \Z \}$
and each row in the top half of $A$, there are either more entries, or the same number of entries
in that row from $z_+ + \Z$ than there are from $z_- + \Z $.

A Borel subalgebra $\b$ of $\g$ is given by the upper triangular matrices in $\g$.  Let $\rho$ be the half sum of the positive roots corresponding to $\b$.
For a weight $\lambda \in \t^*$ let $L(\lambda)$ denote the irreducible highest weight $U(\g)$ with highest weight $\lambda - \rho$.

Given $\bar A \in \sRow(P)$
let $\lambda_{A}$ denote
a weight corresponding to an element of $\bar A$ which is antidominant for $\g_0$, that
is $\langle \lambda_A, \alpha^{\vee} \rangle \not \in \Z_{>0}$ for all positive roots $\alpha$ corresponding to
the Borel subalgebra $\b \cap \g_0$ of $\g_0$.
Note that by \cite[Theorem 2.6]{BV2} the annihilator $\Ann_{U(\g)} L(\lambda_A)$ does not depend on the
choice of $A \in \bar A$ so long as $\lambda_A$ is antidominant for $\g_0$.

\begin{Corollary}
Let $\g = \sp_{2n}$. Then
 $\{ \Ann_{U(\g)} L(\lambda_{A}) \mid \bar A \in \sRow^{\bullet,+}(P)\}$ is a complete set of primitive ideals
  with associated variety equal to $\overline{G \cdot e}$.
\end{Corollary}

\section{Levi subalgebras and an inductive approach}

Throughout this section $\g$ denotes a reductive Lie algebra over $\C$,
and $e$ denotes a nilpotent element of $\g$.
We define $\t^e$ and $\g_0$ as in \cite[\S2]{BGK}.
In \cite[\S3]{BG1} we showed that ``intermediate'' highest weight theories can be defined on
various intermediate Levi subalgebras which lie between $\g_0$ and $\g$;
similar constructions were made by Losev in \cite{Lo3}.
Recall from \cite[\S3]{BG1} that a subalgebra $\s \in \t^e$ is called a {\em full subalgebra}
if $\s$ is the centre of $\g^\s$,
and that
every Levi subalgebra which contains $\g_0$ is equal to $\g^\s$
for some full subalgebra $\s \sub \t^e$.
Given such a full subalgebra $\s$, we choose a parabolic subalgebra $\q_\s$ which contains
$\g^\s$ as a Levi factor.  We let $\{V_\Lambda \mid \Lambda \in \cL_\s\}$ be a complete
set of isomorphism classes of finite dimensional irreducible $U(\g^\s,e)$-modules.
In \cite[\S3.3]{BG2},
we show how a $U(\g^\s,e)$-module $V_\Lambda$ for
$\Lambda \in \cL_\s$ can be induced up to a ``parabolic Verma module'' $M_\s(\Lambda,\q_\s)$  for $U(\g,e)$.
This parabolic Verma module has an irreducible head, denoted by $L_\s(\Lambda, \q_\s)$,
In the case that $\s = \t_e$ we omit the subscript $\s$,
so we write $\cL$ for $\cL_\s$, $\q$ for $\q_s$, and for $\Lambda \in \cL$
we denote the corresponding irreducible highest weight module by $L(\Lambda,\q)$.

Fix $\q$, a choice of a parabolic subalgebra of $\g$ such that $\q$ has $\g_0$ as its Levi factor.
Now $\q^\s$ is a parabolic subalgebra of $\g^\s$ which contains $\g_0$ as its Levi factor.
We denote the
irreducible highest weight module for $U(\g^\s,e)$ corresponding to
$\Lambda$ and $\q^s$ by $L^\s(\Lambda,\q)$.
We note that the labelling of these highest weight modules involves ``shifts'', we refer the reader to
\cite[\S3.3]{BG2} for details.

Let $\s$ be any full subalgebra $\t^e$.
We say that $\s$ is an {\em admissible full
subalgebra for $\q$} if
there exists a parabolic subalgebra $\q_s$ of $\g$ which contains $\q$ and which
has
$\g^\s$ as its Levi factor; note that if $\q_\s$ exists, then it is unique.

Our next lemma shows that we can change highest weight theories
at the level of some intermediate Levi subalgebras.

\begin{Lemma} \label{L:induct}
Let $\q,\q'$ be parabolic subalgebras of $\g$ which have $\g_0$ as a
Levi factor. Let $\s$ be an admissible full subalgebra for both $\q$ and $\q'$ such
that $\q_\s$ and $\q'_\s$ can be chosen so that
$\q_\s = \q'_\s$. Let $\Lambda, \Lambda' \in \cL$ be such that $L^\s(\Lambda, \q) \cong L^\s(\Lambda', \q')$.
Then $L(\Lambda, \q) \cong L(\Lambda', \q')$.
\end{Lemma}

\begin{proof}
This follows immediately from \cite[Proposition 3.6]{BG1}
\end{proof}

We say that $U(\g^\s,e)$ is a rank 1 Levi subalgebra of $U(\g,e)$ if $\s$ is of
codimension
1 in $\t^e$.  This means that the restricted root system of $\g^\s$ with respect
to $\t^e$ is
rank 1 and that there are only two parabolic subalgebras of $\g^\s$ containing
$\g_0$ as
a Levi subalgebra, a positive one and its opposite.  We refer the reader to \cite[\S2]{BruG}
for information on restricted root systems.

Next we define a rank one change of highest weight theories.  Let $\Lambda
\in \cL$, let $\q$
be a parabolic subalgebra with Levi factor $\g_0$ and let $\s$ be an admissible full
subalgebra for $\q$
of codimension 1 in $\t^e$.  Let $\u_\s$ be the nilradical of
$\q_\s$  and let $\q' = (\q')^\s + \u_\s$ where $(\q')^\s$ is the opposite
parabolic subalgebra
of $\g^\s$ to $\q^\s$.    Suppose that $L^\s(\Lambda,\q^\s)$ is finite
dimensional,
then there must exist $\Lambda' \in \cL$ such that $L^\s(\Lambda,\q^\s) \iso
L^\s(\Lambda',\q'^\s)$.
In this case we say that $(\Lambda',\q')$ is obtained from $(\Lambda,\q)$ by a
{\em rank one change} of highest weight theories.
Note that in this case we have that $L(\Lambda,\q) \iso L(\Lambda',\q')$ by
Lemma~\ref{L:induct}.

The following theorem provides an inductive approach to determine if highest
weight $U(\g,e)$-modules are finite dimensional.  Before giving the formal statement
we give an intuitive idea of the essence of the theorem.  It roughly says that
$L(\Lambda,\q)$ is finite dimensional if the corresponding highest weight
module is finite dimensional for each rank 1 Levi subalgebra of $\g$.  The
statement seems complicated, because to check this condition for all rank 1 Levi
subalgebras we have to change highest weight theories.
The theorem is analogous to, though
more complicated than,
the fact that an irreducible highest weight module for $U(\g)$ is finite dimensional
provided the corresponding irreducible highest weight module is finite dimensional for each
$\sl_2$ associated to a simple root.  The proof works by showing that
the inductive condition implies that the highest weight module is also
a lowest weight module and, therefore, must be finite dimensional.

\begin{Theorem} \label{T:simon}
Let $\Lambda \in \mathcal{L}$ and let $\q$ be a parabolic subalgebra of $\g$
with Levi factor $\g_0$.  Then the following are equivalent:
\begin{enumerate}
\item[(i)] $L(\Lambda,\q)$ is finite dimensional.
\item[(ii)] If $(\Lambda',\q')$ is obtained from $(\Lambda,\q)$ by a sequence of
rank one changes of highest weight theories, then
for all full subalgebras $\s$ of $\t^e$ of codimension 1 that are compatible
with $\q'$, we have that $L^\s(\Lambda,(\q')^\s)$ is finite
dimensional.
\end{enumerate}
\end{Theorem}

\begin{proof}
That (i) implies (ii) is clear from \cite[Proposition 3.6]{BG1}.
To prove that (ii) implies (i),
let $\q_-$ be the opposite parabolic subalgebra to $\q$.
The hypothesis ensures that we can pass from $(\Lambda,\q)$ to
$(\Lambda_-,\q_-)$ through a finite
sequences of rank one changes for some $\Lambda_- \in \cL$.  Therefore, we see
that $L(\Lambda,\q) \iso L(\Lambda_-,\q_-)$ by Lemma~\ref{L:induct}.
So $L(\Lambda,\q)$ is both a highest weight and lowest weight module.  Now we
see that the $\t^e$ weights giving nonzero weight spaces are
bounded above and below, and thus finite in number.  As the weight spaces are
finite dimensional, we are done.
\end{proof}

\section{The Robinson-Schensted algorithm}
We use the definitions and notation regarding frames and tables
from \cite[\S4]{BG1}.
The Robinson-Schensted algorithm is used extensively in \cite{BG1}, though the
version presented there only
uses lists of integers.  In this paper we need to use the standard adaption
of this algorithm to lists of complex numbers.

In this paper,
the
Robinson--Schensted algorithm takes as input a word, i.e. a finite list, of
complex numbers, and outputs a tableaux.
Throughout this discussion keep in mind that $\geq$ denotes the partial order on complex
numbers where
$a \ge b$ if $a-b \in \Z_{\ge 0}$.

The Robinson--Schensted algorithm is defined recursively, starting with the empty
table. If $w = a_1 \dots a_n$ is the word, then we assume that
$a_1, \dots, a_{i-1}$ have already been inserted. To insert $a_i$,
assume $b_1, \dots,  b_k$ is the bottom row of the tableaux, and is
non-decreasing.
If $a_i \not < b_j$ for $j=1, \dots, k$, then insert $a_i$ at the end of the
bottom row.
Otherwise there exists $j$ such that $a_i \not < b_{l}$ for $l=1,\dots,j-1$ and $a_i <
b_{j}$. Replace $b_j$ with $a_i$, then recursively insert $b_j$
into the diagram with the bottom row removed.  In this latter
situation we say that $a_i$ {\em bumps} $b_j$. We write $\RS(w)$ for
the output of the Robinson--Schensted algorithm applied to the word
$w$. Given a frame $F$ and $A \in \Tab(F)$ we write $\RS(A) =
\RS(\word(A))$,
where $\word(A)$ is the list of complex numbers obtained by reading off the
entries
of $A$ from left to right, top to bottom.

We now recall two alternative methods for calculating $\part(\RS(w))$
which are
in a sense dual to each other
(recall that $\part(\RS(w))$ denotes the partition corresponding to the shape
of $\RS(w)$).
The first involves finding disjoint non-decreasing subwords of $w$ (a sequence $(a_i)$ of complex numbers
is non-decreasing if $i < j$ implies that $a_i \not > a_j$).
We define $\ell(w,k)$ to be the maximum possible sum of the
lengths of $k$ disjoint non-decreasing subsequences of $w$.
The following lemma is standard for words of integers, see for example \cite[Lemma 4.3]{BG1}.
The statement below is generalized to our setting, and the proof goes
through for this generalization.

\begin{Lemma} \label{L:altRS}
Let $w$ be a word of complex numbers and let $\bp = (p_1 \ge
\dots \ge p_n) = \part(\RS(w))$. Then for all $k \ge 1$, we have
$\ell(w,k) = p_1 + \dots + p_k$.
\end{Lemma}

The ``dual'' version of the above lemma considers lengths of strictly decreasing
subwords.
We define \textctc$(w,k)$ to be the maximum possible sum of the lengths of
$k$ disjoint strictly decreasing subsequences of $w$.
As for Lemma~\ref{L:altRS}, the following lemma is standard for words of integers,
see for example \cite[Lemma 4.4]{BG1}.  Again we state it generalized to words
of complex numbers, and the proof goes through for this case.

\begin{Lemma} \label{L:altcRS}
Let $w$ be a word of complex numbers and let $\bp^T = (p^*_1 \ge \dots \ge
p^*_n)$ be the dual partition to $\bp = \part(\RS(w))$. Then for all
$k \ge 1$, \textctc$(w,k) = p^*_1 + \dots + p^*_k$.
\end{Lemma}

Let $F$ be a frame, and
let $\bar A \in \Row(F)$.  We let
\[
\word(\bar A) =
  \{\word (B) \mid B \in \bar A \text{ and } B \text{ has non-decreasing rows}\}.
\]
Now we let $\RS(\bar A)$ be the row equivalence class
of $\RS(w)$ for some $w \in \word(\bar A)$.
Note that $\RS(\bar A)$, does not depend on the choice of $w \in \word(\bar A)$, so it is well-defined.
We record, without proof, the generalization of \cite[Theorem 4.6]{BG1} to our
current setup.
The proof is essentially the same.
In the theorem, by $\part(\RS(\bar A))$ we mean $\part(B)$ for any $B \in
\RS(\bar A)$.
Recall that a frame is {\em convex} if its left justification has connected columns.

\begin{Theorem} \label{T:recs}
Let $F$ be a convex frame and let $\bar A \in \Row(F)$.  Then $\part(\RS(\bar A))
= \part(F)$
if and only if $\bar A$ contains elements that are justified row equivalent to
column strict.
\end{Theorem}

\section{The component group action, row swapping, and cosets of $\C/\Z$}

Given two s-tables $A$ and $B$, let $A . B$ denote the concatenation of $A$ and
$B$, that is
$A.B$ is the s-table where
each row of $A . B$ is the concatenation of the corresponding rows of $A$ and
$B$.
We also define the concatenation of two row equivalence classes of s-tables in
the obvious way.

Given an s-table $A$ and $z \in \C$, we let $A_{\pm z}$ denote the diagram formed by
taking
the entries in each row of $A$ which belong to $\pm z + \Z$ to form the corresponding
row of $A_{\pm z}$.
So $ A_{\pm z_1} . A_{\pm z_2} . \cdots . A_{\pm z_k} \in \bar A$ where
$\pm z_1 + \Z, \dots, \pm z_k + \Z$ are distinct, and contain all the cosets of $\C/\Z$
present in $A$.
In order for
$A_{\pm z}$ to necessarily be an s-table,
we need to generalize our definition of s-frames to include diagrams with empty rows.
To avoid confusion, we fix a symmetric pyramid $P$ with $2r$ rows as in the introduction.
Now we consider s-frames that have precisely $2r$ rows, some of which may be empty.
We continue to label the rows of s-frames and their related s-tables with
$1, \dots, r, -r, \dots, 1$ from top to bottom, even if the outermost rows happen to be empty.

We require the row swapping operators $\bar s_1 \star , \dots, \bar s_{r-1} \star$ defined
in \cite[\S5.1]{BG2}.
Also, recall the generators $c_1, \dots, c_d$ of $\tilde C$ defined in \cite[\S5.5]{BG2},
where $d$ is the number of distinct odd (even) parts of the Jordan type of $e$
in the case that
$\g = \so_{2n}$ ($\g=\sp_{2n}$) respectively.
Furthermore, an action of $\tilde C$ on a certain subset of $\sRow(P)$ is defined
in the following manner:
Let $\bp = (p_1^2 \ge \dots \ge p_r^2)$ be the Jordan type of $e$.  Let
$p_{i_1} > p_{i_2} > \dots > p_{i_d}$ be the distinct odd (even) parts of $\bp$
in the case that
$\g = \so_{2n}$ ($\g=\sp_{2n}$) respectively, such that $i_k$ is
maximal in $\{ j \mid p_j = p_{i_k} \}$.
Now we define an action of $\tilde C$ on a certain subset of $\sRow(P)$ by declaring that
$ c_j \cdot \bar A = \bar s_{i_j} \star \bar s_{i_j+1} \star \dots \bar s_{r-1} \star (c \cdot (\bar s_{r-1} \star \dots \bar s_{i_j+1} \star \bar s_{i_j} \star \bar A))$,
where $c$ is the operator defined in \cite[\S5.5]{BG2}.
We note that the action of $c$ depends on whether $\g = \so_{2n}$ or $\sp_{2n}$.

We note that the definitions of $\bar s_i \star$ and $c \cdot$ work perfectly well
on diagrams with empty rows. For example, if row $i+1$ of $A$ is empty, then
$\bar s_i \star \bar A$ has row $i$ empty, and row $i+1$ equal to row $i$ of $\bar A$.
Also, if the middle two rows of $\bar A$ are empty, then $c \cdot \bar A = \bar A$.
Thus if $\bar A \in \sRow(P)$ and $z \in \C$, then we can define an action of
$c_j$ on $\bar A_{\pm z}$ via $c_j \cdot \bar A_{\pm z} = \bar s_{i_j} \star \dots \star \bar s_{r-1} \star
   (c \cdot (\bar s_{r-1} \star \dots \star \bar s_{i_j} \star \bar A_{\pm z}))$, provided everything
is defined.

The lemma below follows immediately from the definition of $\bar s_i \star$.

\begin{Lemma} \label{L:split1}
Let $A \in \sRow(F)$ for some s-frame $F$.
\begin{itemize}
\item[(i)]
   $\bar s_i \star \bar A$ is defined if and only if $\bar s_i \star \bar A_{\pm z}$ is defined
for all $z \in \C$.
\item[(ii)]
     If $\bar s_i \star \bar A$ is defined, then
  $$
  \bar s_i \star \bar A =  \bar s_i \star \bar A_{\pm z_1} . \bar s_i \star \bar A_{\pm z_2} . \cdots . \bar s_i \star \bar
A_{\pm z_k},
$$ where
$\pm z_1 + \Z, \dots, \pm z_k + \Z$ are distinct, and contain all the cosets of $\C/\Z$
present in $A$.
\end{itemize}
\end{Lemma}

The action of $\tilde C$ on
$\bar A \in \sRow(P)$ when $L(\bar A,\q)$ is finite dimensional is different when
$\g = \so_{2n}$ from when $\g = \sp_{2n}$, i.e.\ given a word $w$ in the $c_i$'s, we mean
something different by $w \cdot \bar A$ depending on whether
$\g =  \so_{2n}$ or $\g = \sp_{2n}$.
In order to differentiate them, for $w \in \tilde C$, we use $w \dop \bar A$ to denote the action
when $\g = \so_{2n}$ and
$w \cop \bar A$ to denote the action when $\g = \sp_{2n}$.
For notational convenience, we let
$\sRow_+^\diamond(P) = \sRow^\diamond(P)$
    and $\sRow_-^\diamond(P) = \sRow^{\diamond,+}(P)$, where $\sRow^\diamond(P)$
    and $\sRow^{\diamond,+}(P)$ are as defined in the introduction.

The next lemma is clear for the case that $A$ has two rows and then Lemma~\ref{L:split1}
implies the general case.

\begin{Lemma} \label{L:split2}
Let $\g = \so_{2n}$,
let $w \in \tilde C$, and let $A \in \sRow_+^\diamond(P)$.
\begin{itemize}
\item[(i)]
   $w \dop \bar A$ is defined if and only if $w \dop \bar A_{\pm z}$ is defined for
all $z \in \C$.
\item[(ii)]
     If $w \dop \bar A$ is defined, then
  $$w \dop \bar A =  (w \dop \bar A_{\pm z_1}) . (w \dop \bar A_{\pm z_2}) . \cdots . (w \dop \bar A_{\pm z_k}),$$
where
$\pm z_1 + \Z, \dots, \pm z_k + \Z$ are distinct, and contain all the cosets of $\C/\Z$
present in $A$.
\end{itemize}
\end{Lemma}

The lemma below is now a consequence of Lemma~\ref{L:split2} and the definition
of the action of $c$.  It explains how to work out $w \cop \bar A$ from the
rule for $\dop$.  We first consider $\bar{(A^+)}$, as defined in the introduction,
and calculate $w \dop \bar{(A^+)}$. Then if  we  remove from $w \dop \bar{(A^+)}$ the boxes that were added
to $A$ to obtain $A^+$, then the resulting row equivalence class is $w \cop \bar A$.
For this lemma it is necessary to observe that these added boxes are never
moved in the row swapping operations, and they are also not affected by the action of $c$.

\begin{Lemma} \label{L:ctod}
Let $\g = \sp_{2n}$,
let $w \in \tilde C$, and let $A \in \sRow_-^\diamond(P)$.
Then $(w \cop \bar A)^+ = w \dop \bar{(A^+)}$.
\end{Lemma}

In the following lemma $s_{r,-r} \star $ denotes the row swapping operator for {\em tables}
(not s-tables) from \cite[\S4.2]{BG2},
acting on a table which has adjacent rows labeled with $r$ and $-r$.
The lemma is a direct consequence of the
  definitions of $c \cdot$ and $s_{r,-r} \star$.

\begin{Lemma} \label{L:sharpknuth}
    Let $\g = \so_{2n}$, let $\bar A \in \sRow(P)$, and suppose that $\bar A= \bar A_{\pm z}$ for some $z \in \C \setminus \frac{1}{2} \Z$.
\begin{itemize}
 \item[(i)] $s_{r,-r} \star \bar A$ is defined if and only if $c \cdot \bar A$ is defined.
\item[(ii)]
     If
    $s_{r,-r} \star \bar A$ is defined, then
   $s_{r,-r} \star \bar A_z = (c \cdot \bar A)_z$.
\end{itemize}
\end{Lemma}

Our next lemma and the subsequent corollary consider the part of an element of $\sRow_{\phi}^\diamond(P)$
in the upper half-plane,
where $\phi = \pm$.
Recall, from \cite[\S4]{BG1}, that for an s-table $A$,
we write $A_+$ for the subtable of $A$ consisting of the boxes in the upper half-plane.

\begin{Lemma} \label{L:plus}
 Let $\g = \so_{2n}$.
\begin{itemize}
 \item [(i)]
  If $\bar A \in \sRow_+^\diamond(P)$ has all integral or all half-integral entries,
  then
  $(w \dop \bar A)_+$  is justified row equivalent to column strict for all $w \in
   \tilde C$.
 \item [(ii)]
   If $\bar A \in \sRow_+^\diamond(P)$ has
all entries in $\pm z + \Z$ for some $z \in \C \setminus \frac{1}{2}\Z $,
  then
  $(w \dop \bar A)_+$  is justified row equivalent to column strict for all $w \in
   \tilde C$.
\end{itemize}
\end{Lemma}
\begin{proof}
To prove (i), note that by \cite[Theorem 5.13]{BG1}
$L(w \dop \bar A,\q)$ is finite dimensional
    for all $w \in \tilde C$.
    Thus $L^\s( w \dop \bar A, \q^\s)$ is finite dimensional for $\s$ the full
subalgebra equal to
$\langle \sum_{\row(i) = j} f_{i,i} \mid j = 1, \dots, r \rangle$,
    which implies that $(w \dop \bar A)_+$ contains justified row equivalent to column strict elements by
\cite[Theorem 7.9]{BK2}.

To see why (ii) is true,
first recall that
\textctc$(\word(\bar A_z),k) = $
\textctc$(\word(\bar s_i \star \bar A_z),k)$
for any $i$ and $k$.
Now observe that if $B$ is any s-table with entries all in $\pm z + \Z$ for which $c \cdot \bar B$ is defined,
any descending chain in $\word(B_z)$
also occurs in $\word(c \cdot B_z)$.  Thus by Lemma~\ref{L:altcRS} and Theorem~\ref{T:recs} we have that
$w \cdot \bar A_z$ contains justified column strict elements.
Now the fact that $\bar A_+$ must contain justified column strict elements follows from
condition (M3) from the definition of $\sTab^\diamond(P)$.
\end{proof}

\begin{Corollary} \label{C:plus}
   Let $\phi = \pm$ and let $\bar A \in \sRow_\phi^\diamond(P)$.
Then
\begin{itemize}
 \item[(i)] $w \phiop \bar A$ is defined for all $w \in \tilde C$;
\item[(ii)] $(w \phiop \bar A)_+$  is justified row equivalent to column strict for all $w \in
\tilde C$.
\end{itemize}
\end{Corollary}

\begin{proof}
In the case $\phi = +$ this follows from Lemmas~\ref{L:split2} and \ref{L:plus}.
The case $\phi = -$ follows from the case $\phi = +$, and Lemmas~\ref{L:ctod} and
\ref{L:plus}.
\end{proof}

Let $\phi = \pm$ and let $\bar A \in \sRow_\phi^\diamond (P)$.
We use $L(\bar A_0,\q')$ to denote the irreducible $U(\g',e')$-module corresponding
to $\bar A_0$,
where $\g'$ has the same type as $\g$, the parabolic subalgebra
$\q'$ of $\g'$ defined in the same manner as $\q$ is for $\g$, and
$\part(e') = \part(\bar A_0)$.
We use $L_{\operatorname D}(\bar A_{1/2},\q'')$ to denote the irreducible
$U(\g'',e'')$-module
corresponding to $\bar A_{1/2}$,  where $\g''$ is of type D,
the parabolic subalgebra $\q''$ of $\g''$ is defined in the same manner as $\q$ is for $\g$,
and $\part(e'') = \part(\bar A_{1/2})$.

\begin{Lemma} \label{L:intcosets}
Let $\bar A \in \sRow(P)$.
  If $L(\bar A,\q)$ is finite dimensional, then so are $L(\bar A_0,\q')$ and
$L_{\operatorname D}(\bar A_{1/2},\q'')$.
\end{Lemma}
\begin{proof}
   The two row case follows from \cite[Theorem 1.2]{Bro2}.  The general case now
follows from the
two row case, Theorem~\ref{T:simon} and Lemmas~\ref{L:split2} and \ref{L:ctod}.
\end{proof}

\section{Proof of the main theorem} \label{S:main}

In this section we prove Theorem~\ref{T:main}.
To avoid overly cumbersome notation we abuse language somewhat and
say that $\bar A \in \Row(F)$ is {\em row equivalent to column strict} if
the left justification of $\bar A$ contains column strict elements.

First we prove Lemma~\ref{L:mainhelp}, which gives a key step in the proof of Theorem~\ref{T:main}.
For this lemma, we need to recall some notation about changing highest weight theories from \cite{BG2}.
We have defined a parabolic subalgebra in terms of the coordinate pyramid $K$ from the introduction.
We let $W_r$ denote the Weyl group of type $B_r$ with standard generators
$(1,2)(-2,-1), \dots, (r-1,r)(-r,1-r)$ and $(r,-r)$.  Then $W_r$ ``acts'' on $K$ by permuting rows.
For any $\sigma \in W_r$ we can define a parabolic subalgebra $\q_{\sigma}$
with Levi factor $\g_0$ by setting
$\q_{\sigma} = \g \cap \langle e_{i,j} \mid \row_\sigma(i) \leq \row_\sigma(j) \rangle$,
where $\row_\sigma(i)$ denotes the row of $\sigma K$ which contains $i$.

\begin{Lemma} \label{L:mainhelp}
   Let $\phi = \pm$ and let $\bar A \in \sRow_\phi^\diamond(P)$.
   Then $L(\bar A, \q) \cong L(\bar B, \q_{(1,-1)})$ for some $\bar B \in \sRow(P)$.
\end{Lemma}
\begin{proof}
   Let $\bar \omega_r \star$ denote the operator
   $\bar s_{r-1} \star \dots \star \bar s_1 \star$,
     and let $\omega_r$ denote
  $s_{r-1} \dots s_1 \in W_r$, where $s_i = (i,i+1)(1-i, -i)$.
   We also abuse notation somewhat and let $\bar \omega_r^{-1} \star$ denote the operator
   $\bar s_{1} \star \dots \star \bar s_{r-1} \star$.
   By Corollary~\ref{C:plus} we have that $\bar A_+$ is row equivalent to column strict, so
    $\bar \omega_r \star \bar A$
is defined by \cite[Proposition 5.3]{BG2}.  Thus
   Lemma~\ref{L:induct} and \cite[Theorem 4.7]{BG2}  imply that
   $L(\bar A, \q) \cong L(\bar \omega_r \star \bar A, \q_{\omega_r})$.

    We need to consider different cases.  The first case is that
    $P$ has an even number of boxes in row $1$
    if $\g = \so_{2n}$, or $P$ has an odd number of boxes in row $r$ if $\g = \sp_{2n}$.
    In these cases,
     \cite[Theorem 5.11]{BG2} and Lemma~\ref{L:induct} imply that
    $L(\bar A, \q) \cong L(\bar \omega_r \star \bar A, \q_{(r,-r) \omega_r})$,
so
   again by
     Lemma~\ref{L:induct}  we get that
     $L(\bar A, \q) \cong L(\bar \omega_r^{-1} \star \bar \omega_r \star \bar A,
           \q_{\omega_r^{-1} (r,-r) \omega_r}) = L(\bar A,\q_{(1,-1)})$.

     Now we consider the case that
    $P$ has an odd number of boxes in row $1$
    if $\g = \so_{2n}$, or $P$ has an even number of boxes in row $r$ if $\g = \sp_{2n}$.
    By \cite[Theorem 5.11]{BG2} and Lemma~\ref{L:induct} we have that
    $L(\bar A,\q) \cong L(c \cdot (\bar \omega_r \star \bar A), \q_{(r,-r) \omega_r})$.
    Now we need to show that $(c \cdot (\bar \omega_r \star \bar A))_+$ is row equivalent to column strict.
    Indeed, let $a$ be the $\sharp$-element of row $r$ of $\bar \omega_r \star \bar A$ (see \cite[\S5.5]{BG2}
      for the definition of the $\sharp$-element of a list of complex numbers).
     If $a \in \frac{1}{2} \Z$, then $(c \cdot (\bar \omega_r \star \bar A))_+$ is row equivalent to column strict
by the ``integral case'', that is by \cite[Theorem 5.13]{BG1} and Lemma~\ref{L:induct}.
     If $a \not \in \frac{1}{2} \Z$,
        then by condition (M3) in the definition of $\sRow_\phi^\diamond(P)$
        we have that  $(c \cdot (\bar \omega_r \star \bar A)_a)_+$ and
        $(c \cdot (\bar \omega_r \star \bar A)_{-a})_+$  are both convex, so by
        Lemma~\ref{L:sharpknuth} they are both row equivalent to column strict.

Since  $(c \cdot (\bar \omega_r \star \bar A))_+$ is row equivalent to column strict,
$\bar \omega_r^{-1} \star (c \cdot (\bar \omega_r \star A))$ is defined.
Let $\bar B = \bar \omega_r^{-1} \star (c \cdot (\bar \omega_r \bar A))$, so by
     Lemma~\ref{L:induct}  we get that
     $L(\bar A, \q) \cong L(\bar B, \q_{\omega_r^{-1} (r,-r) \omega_r})
         = L(\bar B,\q_{(1,-1)})$.
\end{proof}

Now we are ready to prove Theorem~\ref{T:main}:

\begin{proof}[Proof of Theorem~\ref{T:main}]
  First we suppose that $\bar A \in \sRow_\phi^\diamond(P)$.
We need to show that $L(\bar A,\q)$ is finite dimensional.  To do this,
we show that $L(\bar A,\q)$ is isomorphic to $L(\bar B, \q_-)$, where
$\q_-$ denotes the opposite parabolic to $\q$,
for some $\bar B \in \sRow(P)$.
  We proceed by induction on the number of rows of $P$.

   By Lemma~\ref{L:mainhelp}, $L(\bar A, \q) \cong L(\bar D, \q_{(1,-1)})$
   for some $\bar D \in \sRow(P)$.  We recall that $D_{-2}^2$ is the table
   obtained from $D$ by removing rows 1 and -1, and $P^2_{-2}$ is defined similarly.
We claim that $\bar D_{-2}^2$ is $\tilde C_{-2}^2$ conjugate to an element of
$\sRow_\phi^\diamond(P_{-2}^2)$, where $\tilde C_{-2}^2$ is the component group
corresponding to $P_{-2}^2$, i.e.\ $\tilde C_{-2}^2$ is defined
from $P^2_{-2}$ in the same way that $\tilde C$ is defined from $P$.

It follows from the definition of the row swapping operations and the
action of $c$ that conditions (M1), (M2), (M3) and the convexity part of (M4) hold for $\bar D_{-2}^2$.
So we are left to show that $(\bar D_z)_{-2}^2$ is row equivalent to column strict for all $z \in \C$.

By Lemma~\ref{L:sharpknuth} and Theorem~\ref{T:recs}, we have
that $(\bar D_z)_{-2}^2$ is row equivalent to column strict for all $z \in \C \setminus \frac{1}{2} \Z$.
Since $L_{\operatorname D}(\bar A_{1/2},\q'')$ and $L(\bar A_0,\q')$ are finite dimensional by \cite[Theorem 5.13]{BG1},
we also have that $L_{\operatorname D}(\bar D_{1/2},\q''_{(1,-1)})$ and $L(\bar D_0,\q'_{(1,-1)})$ are finite dimensional.  Here
we are using the notation introduced before Lemma~\ref{L:intcosets}.
Thus by \cite[Theorem 5.11]{BG2}, Lemma~\ref{L:split2}, and \cite[Theorem 5.13]{BG1} we have that
$\bar D_{1/2}$ and $\bar D_0$ are $\tilde C_{-2}^2$ conjugate to a row equivalent to column strict.
This implies that $\bar D$ is also $\tilde C_{-2}^2$ conjugate to a row equivalent to column strict
row equivalence class of s-tables.

So we can repeatedly apply Lemma~\ref{L:mainhelp} to get that
$L(\bar A, \q) \cong L(\bar B, \q_{(1,-1) \dots (r-r)})$ for some $\bar B \in \sRow(P)$.
Now the implication follows on noting that $\q_{(1,-1) \dots (r-r)} = \q_-$, so that
$L(\bar A, \q)$ is both a highest and lowest weight and, hence, finite dimensional.

Now suppose that $L(\bar A,\q)$ is finite dimensional.
We can check that conditions (M1), (M2) and (M3) hold by using the row swapping operations and
\cite[Proposition 3.6]{BG1} along with \cite[Theorem 1.2]{Bro2}.  Also the convexity
of $A_z$ for each $z \in \C$ part of (M4) can be proved in the same way.  So we are left to show that
$A_z$ is row equivalent to column strict for each $z \in \C$.

By Lemma~\ref{L:intcosets}
there exists $w_1, w_2 \in \tilde C$ such that $w_1 \cdot \bar A_{1/2}$ and $w_2 \cdot \bar A_0$
are
justified row equivalent to column strict; furthermore $w_1$ and $w_2$ can be chosen so that
  $w_1 \cdot w_2 \cdot \bar A_{1/2} = w_1 \cdot \bar A_{1/2}$ and
  $w_1 \cdot w_2 \cdot \bar A_{0} = w_2 \cdot \bar A_{0}$.
This allows as to assume that $\bar A_{1/2}$ and $\bar A_0$ are row equivalent to column strict.

So we need to consider $A_{\pm z}$
for $z \not \in \frac{1}{2}\Z$.
Let $p_1, \dots, p_r, p_r, \dots, p_1$ be the row lengths of $A_{\pm z}$.
We may assume that $p_1 \neq 0$, since otherwise we are done by induction.

Now we consider $\bar A_z$.  Let $p_1', \dots, p_r', p_r'', \dots, p_1''$ be the row lengths of $A_z$.
Furthermore, by changing highest theories using the row swapping operations as in \cite[\S4]{BG2}
and using Lemmas~\ref{L:sharpknuth} and \ref{L:ctod}, we may assume that
$p'_i = \lfloor \frac{p_i}{2} \rfloor$
and $p''_i = \lceil \frac{p_i}{2} \rceil$ for all $i$.
We choose $A_z \in \bar A_z$ to have increasing rows.

By induction, we have that $(A_z)_{-2}^2$ is justified row equivalent to column strict.
Since $L(\bar A,\q)$ is finite dimensional, we also have that $(A_z)_+$ and $(A_z)_-$ are justified row equivalent to column strict, here $(A_z)_-$ denotes
the bottom half of $A_z$.  So
if we slide row $-1$ of the left justification $l(A_z)$ of $A_z$ right so that it is right justified with row $-2$ of $l(A_z)$, then
the resulting diagram is row equivalent to column strict.  Let $B$ denote this diagram.  So by Lemma~\ref{L:altcRS}.
the column lengths of $B$ form a lower bound for the transpose $\part(\RS(A_z))^T$ of $\part(\RS(A_z))$.
We explicitly give this bound in all cases:

If $p_1$ and $p_2$ are even and $p_1' \leq p_2'/2$, then
\[
\part(\RS(A_z))^T \geq
     ((r-1)^{2 p_1'}, (r-2)^{p_2' - 2p_1'}, (r-4)^{p_3'-p_2''},
       (r-5)^{p_3''-p_3'}, \dots, 2^{p_r'-p_{r-1}''}, 1^{p_r''-p_r'}),
\]
hence,
\[
\part(\RS(A_z)) \leq
(p_r'',p_r', \dots, p_2'',p_2',2 p_1').
\]

If $p_1$ and $p_2$ are even and $p_1' > p_2'/2$, then
\[
\part(\RS(A_z))^T \geq
     (r^{2p_1'-p_2'}, (r-1)^{2p_2' - 2 p_1'}, (r-4)^{p_3'-p_2''},
       (r-5)^{p_3''-p_3'}, \dots, 2^{p_r'-p_{r-1}''}, 1^{p_r''-p_r'}),
\]
hence,
\[
\part(\RS(A_z)) \leq
(p_r'',p_r', \dots, p_2'',p_2',p_2', 2 p_1'-p_2').
\]

If $p_1$ is even and $p_2$ is odd and $p_1' \leq p_2'/2$, then
\[
\part(\RS(A_z))^T \geq
     ((r-1)^{2 p_1'-1}, (r-2)^{p_2'-2p_1'+2}, (r-4)^{p_3'-p_2''},
       (r-5)^{p_3''-p_3'}, \dots, 2^{p_r'-p_{r-1}''}, 1^{p_r''-p_r'}),
\]
hence,
\[
\part(\RS(A_z)) \leq
(p_r'',p_r', \dots, p_2'',p_2'+1, 2 p_1'-1).
\]

If $p_1$ is even and $p_2$ is odd and $p_1' > p_2'/2$, then
\begin{align*}
\part(\RS&(A_z))^T \geq \\
     &(s^{2p_1'-p_2'-1}, (s-1)^{2p_2' - 2 p_1'+1}, s-2, (s-4)^{p_3'-p_2''},
       (s-5)^{p_3''-p_3'}, \dots, 2^{p_s'-p_{s-1}''}, 1^{p_s''-p_s'}),
\end{align*}
hence,
\[
\part(\RS(A_z)) \leq
(p_s'',p_s', \dots, p_2'',p_2'+1, p_2', 2 p_1'-p_2'-1).
\]

If $p_1$ is odd and $p_2$ is even and $p_1'' \leq p_2'/2$, then
\[
\part(\RS(A_z))^T \geq
     ((s-1)^{p_1'+p_1''}, (s-2)^{p_2'-p_1'-p_1''}, (s-4)^{p_3'-p_2''},
       (s-5)^{p_3''-p_3'}, \dots, 2^{p_s'-p_{s-1}''}, 1^{p_s''-p_s'}),
\]
hence,
\[
\part(\RS(A_z)) \leq
(p_s'',p_s', \dots, p_2'',p_2', p_1'+p_1'').
\]

If $p_1$ is odd and $p_2$ is even and $p_1'' > p_2'/2$, then
\[
\part(\RS(A_z))^T \geq
     (s^{2p_1'-p_2'+1}, (s-1)^{2p_2' - 2 p_1'-1}, (s-4)^{p_3'-p_2''},
       (s-5)^{p_3''-p_3'}, \dots, 2^{p_s'-p_{s-1}''}, 1^{p_s''-p_s'}),
\]
hence,
\[
\part(\RS(A_z)) \leq
(p_s'',p_s', \dots, p_2'',p_2', p_2', p_1'+p_1''-p_2').
\]

If $p_1$ and $p_2$ are odd and $p_1'' \leq (p_2''+1)/2$, then
\begin{align*}
\part(\RS&(A_z))^T \geq \\
     &((s-1)^{p_1'+p_1''-1}, (s-2)^{p_2'-p_1'-p_1''+2}, (s-4)^{p_3'-p_2''},
       (s-5)^{p_3''-p_3'}, \dots, 2^{p_s'-p_{s-1}''}, 1^{p_s''-p_s'}),
\end{align*}
hence
\[
\part(\RS(A_z)) \leq
(p_s'',p_s', \dots, p_2'',p_2'+1, p_1'+p_1''-1).
\]

If $p_1$ and $p_2$ are odd and $p_1'' > (p_2''+1)/2$, then
\[
\part(\RS(A_z))^T \geq
     (s^{2p_1'-p_2'}, (s-1)^{2p_2' - 2 p_1'}, s-2,(s-4)^{p_3'-p_2''},
       (s-5)^{p_3''-p_3'}, \dots, 2^{p_s'-p_{s-1}''}, 1^{p_s''-p_s'}),
\]
hence
\[
\part(\RS(A_z)) \leq
(p_s'',p_s', \dots, p_2'',p_2'+1, p_2', p_1'+p_1''-p_2'-1).
\]

Note that $(\bar A_z)_+$ defines an irreducible highest weight module for
$U(\gl_{m'},e')$, where $m' = p_1'+\dots+p_r'$ and $e'$ is the
nilpotent element corresponding to the frame of $(A_z)_+$.
Since $(\bar A_z)_+$ is row equivalent to column strict, we have that this
highest weight module is finite dimensional by \cite[Theorem 7.9]{BK2}.  We can
argue similarly with $(A_z)_-$ in place of $(A_z)_+$.
We can perform row swaps on $\bar A_z$ to obtain $\bar D$, a row equivalence class of s-tables with
rows lengths $(p_2', p_3', \dots, p_s', p_1', p_1'', p_s'', \dots, p_2'')$; let $\omega \in W_r$ correspond
to this sequence of row swaps.
Since $\bar D_+$ is convex, the corresponding highest weight module for $U(\gl_{m'},e')$ is finite dimensional.
Therefore, by Theorem~\ref{T:recs}, we have that
$\bar D_+$ is row equivalent to column strict. Similarly $\bar D_-$ is row equivalent to column strict.

Consider $\omega \star \bar A$, the row equivalence obtained from $\bar A$ by the row swaps
used to get $\bar D$ from $\bar A_z$.  Then $L(\omega \star A, \q_\omega) \iso L(A,\q)$ is finite
dimensional.  So it follows from \cite[Proposition 3.6]{BG1} and \cite[Theorem 1.2]{Bro2}
that $\bar D_{-r}^r$ is row equivalent to column strict.

Let $D \in \bar D$.  It follows that in $D$ we can find $p_1'$ decreasing chains of length $2 r$.
This implies that $\part(\RS(\bar D))$ is smaller than a partition
with $2 r$ parts which ends in $p_1'$.
In all cases, except when $p_1$ and $p_2$ are odd, this guarantees that
\[
\part(\RS(A_z)) =
(p_r'',p_r', \dots, p_2'',p_2', p_1'', p_1'),
\]
which by Theorem~\ref{T:recs} implies that $\bar A_z$ is row equivalent to column strict.

Now we only need consider the case that $p_1$ and $p_2$ are odd.
Let $d$ be the $\sharp$-element of row $-r$ of $D$ (see \cite[\S5.5]{BG2}).
In this case, we can find $p_1''$ descending chains in $D_-$ of length $r$.
We set aside one of these chains which begins with $d$, and let $c_1, \dots, c_{p_1'}$ be the remaining chains.
We can also find $p_1''$ descending chains in $(c \cdot D)_+$ of length $s$, and
we set aside one of these chains which ends with $-d$, and let $d_1, \dots, d_{p_1'}$ be the remaining chains.
Let $D'$ denote $D$, but with one occurrence of $d$ removed from row $-s$.
So $(D')_{-r}^r$ is justified row equivalent to column strict, thus it is possible to
glue each chain $c_i$ with a chain $d_j$ to obtain $p_1'$ descending chains in $D$ of length $2 r$.
We can also combine the chains of length $r$ that we ``set aside'', then remove one occurrence of $d$,
to obtain a chain of length $2r-1$ in $D$
which is disjoint from the chains of length $2 r$ we just constructed.
This implies that $\part(\RS(D))$ is larger than a partition
with $2 r$ parts of the form $(*, p_1'', p_1')$.
So
\[
\part(\RS(A_z)) =
(p_s'',p_s', \dots, p_2'',p_2', p_1'', p_1'),
\]
which by Theorem~\ref{T:recs}  implies that $A_z$ is justified row equivalent to column strict.
\end{proof}

As a corollary to the theorem we obtain a parametrization of the primitive ideals of $U(\g)$ with associated variety
$\overline {G \cdot e}$.
For every element of the set $\{\pm z  + \Z \mid z \in \C \setminus \frac{1}{2} \Z\}$ make a choice
$z_+ + \Z$ and $z_- + \Z$ so that
$\pm z + \Z = (z_+ + \Z) \sqcup (z_+ + \Z)$.
Let $\sTab_\phi^{\bullet}(P)$ denote the set of elements $A \in \sTab_\phi^{\diamond}(P)$
such that
for each element of the set $\{\pm z + \Z \mid z \in \C \setminus \frac{1}{2} \Z \}$
and each row in the top half of $A$, there are either more entries, or the same number of entries
in that row from $z_+ + \Z$ than there are from $z_- + \Z$.

In the case $\g = \so_{2n}$ we set
$C = C_G(e)/C_G(e)^\circ$ where $G = \SO_{2n}$.
Then $C = \tilde C = 1$ if all parts of $\part(P)$ are even, and
$C \sub \tilde C$ is a subgroup of index 2 otherwise.  In the latter case
we use $\sRow_\phi^{\bullet '}(P)$
to denote the elements $A$ of $\sRow_\phi^\bullet(P)$
for which there exists $w \in \tilde C \setminus C$
such that $w \cdot \bar A = \bar A$.
We note that
if $\bar A$ has no boxes filled with $0$, then
$\bar A \notin \sRow_\phi^{\bullet '}(P)$.

The following corollary is immediate from Theorem~\ref{T:main} and
the map $\cdot ^\dagger$ from \cite[\S1.2]{Lo2}.

\begin{Corollary} \label{C:prim} $ $
\begin{enumerate}
\item[(a)] Let $\g = \sp_{2n}$. Then
\[
\{\Ann_{U(\g)}L(\lambda_{\bar A}) \mid \bar A \in \sRow_-^{\bullet}(P) \}
\]
is a complete set of
pairwise distinct
primitive ideals of $U(\g)$
with associated variety $\bar{G \cdot e}$.
\item[(b)] Let $\g = \so_{2n}$.
\begin{enumerate}
\item[(i)]  Suppose that all parts of $\part(P)$ are even.  Then
\[
\{\Ann_{U(\g)}L(\lambda_{\bar A}) \mid \bar A \in \sRow_+^{\bullet}(P) \}
\]
is a complete set of
pairwise distinct primitive ideals of $U(\g)$ with
associated variety $\bar{G \cdot e}$.
\item[(ii)] Suppose that $\part(P)$ has odd parts.  Then
\begin{align*}
\{\Ann_{U(\g)} & L(\lambda_{\bar A})   \mid \bar A \in \sRow_+^{\bullet '}(P) \\
&\cup
\{\Ann_{U(\g)}L(\lambda_{\bar A})  \mid \bar A \in \sRow_+^\bullet(P) \setminus  \sRow_+^{\bullet '}(P) \} \\
&\cup
\{\Ann_{U(\g)}L(\lambda_{c_1 \cdot \bar A})  \mid \bar A \in
\sRow_+^\bullet(P) \setminus  \sRow_+^{\bullet '}(P) \}
\end{align*}
is a complete set of
pairwise distinct primitive ideals of $U(\g)$ with
associated variety $\bar{G \cdot e}$.
\end{enumerate}
\end{enumerate}
\end{Corollary}

\end{document}